\documentclass[12pt]{amsart}
\usepackage{amsthm,amsmath,amssymb}
\usepackage{enumitem}
\usepackage[all]{xy}

\usepackage{orcidlink}

\newtheoremstyle{theorem}
     {11pt}
     {11pt}
     {}
     {}
     {\bfseries}
     {}
     {.5em}
     {\noindent\thmnumber{#2}. \thmname{#1} \thmnote{(#3)}}

\theoremstyle{theorem}

\newtheorem{thm}{Theorem}

\newtheorem{lemma}[thm]{Lemma}
\newtheorem{propo}[thm]{Proposition}

\newtheorem{coro}[thm]{Corollary}
\newtheorem{ques}[thm]{Question}

\newcommand{\R}{\mathbb{R}}

\newcommand{\cl}[2][X]{\mathrm{cl}_{#1}\!\left(#2\right)}
\newcommand{\ZFC}{\textsf{ZFC}}

\begin{document}

\title[Functional countability vs exponential separability]{Comparing functional countability and exponential separability}

\author[R. Hernández-Gutiérrez]{Rodrigo Hernández-Gutiérrez\orcidlink{0000-0002-5949-0871}}
\address[R. Hernández-Gutiérrez]{Departamento de Matemáticas, Universidad Autónoma Metropolitana campus Iztapalapa, Av. San Rafael Atlixco 186, Leyes de Reforma 1a Sección, Iztapalapa, 09310, Mexico city, Mexico}
\email[R. Hernández-Gutiérrez]{rod@xanum.uam.mx}

\author[S. Spadaro]{Santi Spadaro\orcidlink{0000-0001-5880-8799}}
\address[S. Spadaro]{Department of Engineering \\
University of Palermo \\
viale delle Scienze 8 \\
90128 Palermo, Italy}
\email{santidomenico.spadaro@unipa.it}

\makeatletter
\@namedef{subjclassname@2020}{\textup{2020} Mathematics Subject Classification}
\makeatother

\begin{abstract}
A space is functionally countable if every real-valued continuous function has countable image. A stronger property recently defined by Tkachuk is exponential separability. We start by studying these properties in GO spaces, where we extend results by Tkachuk and Wilson, and prove a conjecture by Dow. We also study some subspaces of products that are functionally countable and the influence of the $G_\delta$-topology on exponential separability. Finally, we give some examples of functionally countable spaces that are separable and uncountable.
\end{abstract}

\maketitle 
\section{Introduction}

All spaces under discussion will be assumed to be Hausdorff.

Let $X$ be a topological space. Given $\mathcal{F}\subset\wp(X)$, a set $A\subset X$ is \emph{strongly dense} in $\mathcal{F}$ if $A\cap\left(\bigcap\mathcal{G}\right)\neq\emptyset$ whenever $\mathcal{G}\subset\mathcal{F}$ and $\bigcap G\neq\emptyset$. $X$ is \emph{exponentially separable} if, for every for countable family $\mathcal{F}$ of closed sets, there is countable subset $A$ of $X$ that is strongly dense in $\mathcal{F}$.

Exponentially separable spaces are a subclass of \emph{functionally countable spaces}, that is, spaces whose every continuous real-valued image is countable. Indeed, functionally countable spaces were characterized by Tkachuk in \cite[Theorem 3.8]{tka-nice_subclass} as those spaces which satisfy the definition of exponential separability when $\mathcal{F}$ is restricted to families of zero-sets. Moreover, these two notions coincide in a few notable cases, like perfectly normal (\cite[Corollary 3.9]{tka-nice_subclass}) and countably compact normal spaces (\cite[Theorem 3.30]{tka-nice_subclass}).

In this paper, we study these two classes of spaces. Our main results can be summarized as follows:
\begin{enumerate}[label=(\arabic*)]
\item We characterize the scattered GO spaces and the locally countable GO spaces that are exponentially separable  (Corollary \ref{coro:GO_exp_sep}); this extends a previous characterization by Tkachuk and Wilson (\cite{tka_wil-fc_GO}).
\item We show that for Souslin lines functional countability is equivalent to exponential separability (Proposition \ref{propo:souslin}).
\item We prove a conjecture of Alan Dow's by showing that any Souslin line defined by lexicographically ordering a normal Souslin tree is functionally countable (Theorem \ref{thm:souslin_line_fc}).
\item We give an example of a $\sigma$-closed discrete space that is functionally countable but not exponentially separable because it has arbitrarily large extent (Theorem \ref{uspenskij-thm}).
\item We give an example of an uncountable Tychonoff space which is separable and exponentially separable (Theorem \ref{ex:unctble_sep_exp_sep}).
\item We prove that any Ostaszewki space is an example of an uncountable, perfectly normal, hereditarily separable and exponentially separable space (Theorem \ref{thm:ostaszewski}).
\end{enumerate}

\section{GO spaces}

Recall that a LOTS (linearly ordered topological space) is a space such that its topology is the order topology defined by a linear order. A GO (generalized ordered) space can be defined as a space that can be embedded as a subspace of a LOTS. In this section, we study some classes of GO spaces which are known to be functionally countable, and prove that they are in fact exponentially separable. 

We start by stating the following two facts about subspaces of ordinals that have countable extent.

\begin{lemma}\label{lemma:char-ctble-ext-ordinals}
Let $\lambda$ be an ordinal and $X\subset\lambda+1$. Then the following are equivalent
\begin{enumerate}[label=(\roman*)]
    \item $X$ has countable extent, and
    \item for every $\xi\leq\lambda$ with $\mathsf{cof}(\xi)=\omega_1$, either
    \begin{enumerate}
        \item there is $\eta<\xi$ such that $X\cap \xi=X\cap\eta$ or
        \item if $X\cap\xi$ is not stationary, then $\xi\in X$.
    \end{enumerate}
\end{enumerate}
\end{lemma}
\begin{proof}
First, we prove that $\neg(ii)$ implies $\neg(i)$. Let $\xi\leq\lambda$ be such that $X\cap\xi$ is unbounded nonstationary in $\xi$, $\mathsf{cof}(\xi)=\omega_1$ but $\xi\notin X$. Let $C\subset\xi$ be a club such that $C\cap X=\emptyset$. We may assume that $C$ is of order type $\omega_1$. Recursively choose an increasing sequence $S=\{x_\alpha\colon\alpha<\omega_1\}\subset X$ such that $(x_{\alpha},x_{\alpha+1})\cap C\neq\emptyset$ for all $\alpha<\omega_1$. Any limit point of $S$ is contained in $\overline{C}$ and $\overline{C}=C\cup\{\xi\}$. Since $(C\cup \{\xi\})\cap X=\emptyset$, then $S$ is discrete and uncountable in $X$.

Now, we prove that $\neg(i)$ implies $\neg(ii)$. Let $S$ be an uncountable closed discrete subset of $X$; we may assume that $S$ is of order type $\omega_1$. Let $\xi=\sup S$, then $\mathsf{cof}(\xi)=\omega_1$. The set $F=S'$ of limit points of $S$ is a closed unbounded subset of $\xi$ which misses $X\cap\xi$. Thus, $X\cap\xi$ is nonstationary in $\xi$ but $\xi\notin X$. 
\end{proof}

\begin{lemma}\label{lemma:ctle-ext-ordinals-unctable-cof}
Let $\lambda$ be an ordinal and let $X\subset\lambda+1$ have countable extent. Assume that $\xi\leq\lambda$ is a limit ordinal such that $\mathsf{cof}(\xi)>\omega_1$ and $\sup(X\cap\xi)=\xi$. Then $X\cap\xi$ is stationary in $\xi$.
\end{lemma}
\begin{proof}
If $X\cap\xi$ were not stationary in $\xi$, then there would be a closed unbounded set $C \subset \xi$ such that $C \cap X=\emptyset$. It is possible to construct a strictly increasing sequence $S=\{x_\alpha: \alpha < \omega_1\} \subset X$ such that $(x_\alpha, x_{\alpha+1}) \cap C \neq \emptyset$ for all $\alpha<\omega_1$. Since $\mathsf{cof}(\xi)>\omega_1$, $S$ is bounded in $\xi$. Then all limit points of $S$ would be inside $C$ and therefore $S$ would be a closed discrete subset of $X$, contradicting countable extent.
\end{proof}

Next, we show that subspaces of ordinals with countable extent are exponentially separable and then generalize this result to scattered GO spaces. 

\begin{thm}\label{thm:subspaces-ordinals}
    Let $X$ be a space of countable extent such that $X \subset \lambda+1$, for some ordinal $\lambda$. Then $X$ is exponentially separable.
\end{thm}
\begin{proof}

Define $X_\alpha=X \cap (\alpha+1)$, for every $\alpha \leq \lambda$. Inductively we prove that $X_\alpha$ is exponentially separable for all $\alpha\leq\lambda$.

If $\alpha\leq\min X$ then $X_\alpha$ clearly is exponentially separable. Now, suppose we have proved that $X_\beta$ is exponentially separable, for every $\beta<\gamma$ and some $\gamma\leq \lambda$.

If $\gamma$ is a successor ordinal or $\gamma$ has countable cofinality, then $X_\gamma$ is exponentially separable, since it is a countable union of exponentially separable spaces (see \cite[Proposition 3.2 (d)]{tka-nice_subclass}). 

It remains to argue what happens if $\gamma$ has uncountable cofinality. If there exists $\alpha<\gamma$ such that $X \cap \gamma\subset\alpha+1$, then $X_\gamma=X_\alpha$ or $X_\gamma=X_\alpha\cup\{\gamma\}$. In either case, $X_\gamma$ is exponentially separable by the induction hypothesis.  Now, assume that $X \cap \gamma$ is unbounded in $\gamma$. Let $\{F_n: n < \omega \}$ be a sequence of closed subsets of $X_\gamma$. Let's consider two cases.

\begin{enumerate}[label={\roman*.}]

\item $\gamma\in X$. Let $A=\{n<\omega\colon \gamma\in F_n\}$ and $B=\omega\setminus A$. Notice that there exists $\alpha<\gamma$ such that $F_n\subset\alpha+1$ for each $n\in B$. Since we are inductively assuming that $X_\alpha$ is exponentially separable, we know that there exists a countable set $D\subset X_\alpha$ which is strongly dense in $\{F_n\cap (\alpha+1)\colon n\in\omega\}$. Then $C=D \cup \{\gamma\}$ is a countable set which is strongly dense in $\{F_n: n < \omega \}$, which proves that $X_\gamma$ is exponentially separable.

\item $\gamma\notin X$. By Lemma \ref{lemma:char-ctble-ext-ordinals} or Lemma \ref{lemma:ctle-ext-ordinals-unctable-cof} we know that $X\cap\gamma$ is a stationary subset of $\gamma$. In this case, let $A=\{n<\omega\colon \sup(F_n\cap\gamma)=\gamma\}$ and $B=\omega\setminus A$. Notice that $\{\cl[\gamma]{F_n\cap\gamma}\colon n\in A\}$ is collection of clubs in $\gamma$, so the intersection of this collection and $X$ is nonempty. Thus, there exists
$$\chi\in (X\cap\gamma)\cap\left(\bigcap\{F_n\colon n\in A\}\right).$$ Since all elements of $\{F_n\colon n\in B\}$ are bounded in $\gamma$, we may assume that $\chi$ is such that $F_n\subset\chi+1$ for all $n\in B$. Since we are inductively assuming that $X_\chi$ is exponentially separable, it follows that there is a countable set $C\subset X_\chi$ which is strongly dense in $\{F_n\cap (\chi+1)\colon n<\omega\}$. It can be easily proved that $C$ is also strongly dense in $\{F_n\colon n<\omega\}$.
\end{enumerate}
\end{proof}

\begin{thm}
Let $X$ be a scattered GO space of countable extent. Then $X$ is exponentially separable.
\end{thm}
\begin{proof}
By \cite[Theorem]{purisch-compactifications}, $X$ is in fact a LOTS (possibly, with a different order relation) and by \cite[Theorem 3]{purisch-orderability_closed_images}, there is an ordinal $\lambda$, $Y\subset\lambda$  and a function $f\colon Y\to X$ that is closed, continuous, surjective and has fibers of cardinality at most $2$.

We claim that $Y$ has countable extent. Indeed, let $D\subset Y$ be closed and discrete. Let $E=f[D]$. Since $f$ is closed, any subset of $E$ is closed in $X$. Then $E$ is closed discrete. Since $Y$ has countable extent, $E$ is countable. Since the fibers of $f$ are finite, then $D$ is also countable.

By Theorem \ref{thm:subspaces-ordinals}, $Y$ is exponentially separable. As noted in \cite[Proposition 3.2 (c)]{tka-nice_subclass}, exponentially separability is preserved under continuous images. Thus, $X$ is exponentially separable.
\end{proof}

In \cite{tka_wil-fc_GO} another subclass of GO spaces shown to be functionally countable is that of locally countable GO spaces of countable extent; we can extend that result as follows.

\begin{thm}
Let $X$ be a GO space which is locally countable. Then $X$ is exponentially separable if and only if $X$ has countable extent.
\end{thm}

\begin{proof}
First, assume that $X$ is a GO space and $D$ is an uncountable, closed and discrete subset of $X$. Then there is a function $f\colon D\to\R$ that is injective on some uncountable subset of $D$. Since $X$ is normal, $f$ can be extended to $X$ by the Tietze-Urysohn theorem. Thus, $X$ is not functionally countable. Therefore, the direct implication follows.

To prove the reverse direction, let $X$ be a GO space of countable extent.

Let $\sim$ be the following equivalence relation on $X$: $x \sim y$ if and only if $[\min{ \{x,y\}}, \max{\{x,y\}}]$ is countable. Equivalence classes are clearly convex, and by local countability, they are clopen. The cofinality and the coinitiality of every equivalence class is $\leq \omega_1$. So if $C$ is an equivalence class then $C$ is order-isomorphic to one of the following:

\begin{enumerate}
\item A subset of $\mathbb{Q}$ (if its cofinality and coinitiality are $\leq \omega$).
\item A subset of the lexicographically ordered  $\omega_1 \times\mathbb{Q}$ (if its cofinality is $\omega_1$ and its coinitiality is $\leq \omega$).
\item A subset of the lexicographically ordered $\omega_1^* \times\mathbb{Q}$ (if its cofinality is $\leq \omega$ and its coinitiality is $\omega_1$).
\item A set of the form $B\cup A$ where $A$ is as in case (2), $B$ is as in case (3) and $b<a$ for all $a\in A$ and $b\in B$ (if its cofinality and coinitiality are both $\omega_1$).
\end{enumerate}

We claim that each equivalence class is a countable union of closed scattered subspaces. Let $C$ be an equivalence class; we will describe closed scattered sets $\{C_n\colon n\in\omega\}$ whose union is $C$. We will only do this for case (2) described above and leave the rest to the reader. Since the cofinality of $C$ is $\omega_1$, there is $S=\{s_\alpha\colon \alpha<\omega_1\}\subset C$  such that $s_\alpha<s_\beta$ if $\alpha<\beta<\omega_1$ and $S$ is cofinal in $C$. For each $\alpha<\omega_1$ let $J_\alpha=\{x\in C\colon \forall \beta<\alpha, s_\beta<x\leq s_\alpha\}$. By choosing a subsequence of $S$, we may assume that $J_\alpha$ is countable infinite for each $\alpha\in\omega_1$, so we may choose an enumeration $J_{\alpha}=\{x(\alpha,n)\colon n\in\omega\}$ for each $\alpha\in\omega_1$. Also, let $L$ be the set of limit points of $S$; that is, $x\in L$ if and only if there is a limit $\beta\in\omega_1$ such that $x=\sup\{s_\alpha\colon\alpha<\beta\}$. Then let $C_n=\{x(\alpha,n)\colon \alpha\in\omega_1\}\cup L$ for each $n\in\omega$.

Since every equivalence class is clopen, it follows that $X$ itself is a countable union of closed scattered subspaces. Now, since every scattered GO space of countable extent is exponentially separable then $X$ is also exponentially separable.
\end{proof}

We summarize our results so far as follows; notice that this extends \cite[Corollary 3.6]{tka_wil-fc_GO} and \cite[Corollary 3.8]{tka_wil-fc_GO}.

\begin{coro}\label{coro:GO_exp_sep}
Let $X$ be a GO space that is either scattered or locally countable. Then the following are equivalent.
\begin{enumerate}[label=(\roman*)]
    \item $X$ is functionally countable,
    \item $X$ is exponentially separable, and
    \item $X$ has countable extent.
\end{enumerate}
\end{coro}

Recall that a Souslin line is a linearly ordered set that (with the order topology) is ccc but not separable. In \cite{GD-HG} it was shown that if Souslin lines exist, then there are some of them that are functionally countable and some that are not. Here we show that those that are functionally countable are in fact exponentially separable.

\begin{propo}\label{propo:souslin}
If $X$ is a Suslin line, then $X$ is functionally countable if and only if $X$ is exponentially separable.
\end{propo}
\begin{proof}
By \cite[Corollary 3.9]{tka-nice_subclass}, we only need to prove that $X$ is perfectly normal. Since $X$ is regular and hereditarily Lindelöf, every open set $U$ of $X$ is the union of a countable family of open sets whose closure is contained in $U$. Thus, every open subset of $X$ is an $F_\sigma$ and consequently, every closed subset of $X$ is a $G_\delta$. Since $X$ is orderable, it is normal. Thus, $X$ is perfectly normal.
\end{proof}

According to \cite[Lemma 20]{GD-HG}, if a Souslin line is densely ordered and is left-separated in type $\omega_1$, it is functionally countable. From the previous result we obtain the following.

\begin{coro}
Every densely ordered and left-separated Suslin line in type $\omega_1$ is exponentially separable.
\end{coro}

During the 2023 BIRS-CMO Set-Theoretic Topology workshop, Alan Dow conjectured that Souslin lines obtained by lexicographically ordering a Souslin tree are always functionally countable. We were able to prove this conjecture by using that a normal Souslin tree with the tree topology is functionally countable, a fact previously     proved by Steprāns in \cite{steprans-trees}. 

For terminology on trees, the reader may consult \cite[Section III.5]{kunen-set_theory_2011} or \cite{todorcevic-handbook}. For the reader's convenience, we recall some facts and fix notation. 

Let $\langle T,\sqsubseteq\rangle$ be a tree. For an ordinal $\alpha$, the $\alpha$-th level of $T$ is the set $T_\alpha$ of all elements of $T$ with height $\alpha$ . The height of $T$ is the minimal $\beta$ such that $T_\beta=\emptyset$. An $\omega_1$-tree is a tree of height $\omega_1$ all levels of which are countable. A Souslin tree is an $\omega_1$-tree without countable chains or antichains.

Given $x\in T$, let $x{\downarrow}=\{y\in T\colon y\sqsubset x\}$ and $x{\uparrow}=\{y\in T\colon x\sqsubset y\}$. A tree $T$ of height $\kappa$ is called \emph{normal} if
\begin{enumerate}[label=(\roman*)]
\item $T$ is \emph{rooted}: $\lvert T_0\rvert=1$ ,
\item $T$ is \emph{splitting}: for all $\alpha+1<\kappa$ and $x\in T_\alpha$ there are $y_0,y_1\in T_{\alpha+1}$ with $x\sqsubset y_0$, $x\sqsubset y_1$ and $y_0\neq y_1$ ,
\item $T$ is \emph{pruned}: if $\alpha<\beta<\kappa$ and $x\in T_\alpha$, there exists $y\in T_\beta$ such that $x\sqsubset y$, and
\item $T$ is \emph{Hausdorff}: if $\alpha<\kappa$ is a limit and $x,y\in T_\alpha$ are such that $x{\downarrow}=y{\downarrow}$, then $x=y$ .
\end{enumerate}

The \emph{tree topology} on $T$ is the topology generated by the sets of the form $x{\downarrow}$ and $x{\uparrow}$, for $x\in T$. Notice that the tree topology of a normal tree is Hausdorff. In \cite{steprans-trees} it is proved that a normal uncountable tree is functionally countable if and only if it is a Souslin tree (although the term ``functionally countable'' is not mentioned).

Given a normal tree $\langle T,\sqsubseteq\rangle$ and a total order $\preceq$ in $T$, we define a total order $\leq$ in $T$ as follows. Given $s,t\in T$, we write $s\perp t$ if $s\not\sqsubseteq t$ and $t\not\sqsubseteq s$. If $s,t\in T$ and $s\perp t$, the set $s{\downarrow}\cap t{\downarrow}$ has a maximal element $s\wedge t$ (because of normality of the tree) so there are $s_{s\wedge t},t_{s\wedge t}\in T$ that are immediate successors of $s\wedge t$, $s_{s\wedge t}\sqsubseteq s$ and $t_{s\wedge t}\sqsubseteq t$.  If $t,s\in T$ then define $s\leq t$ if
\begin{itemize}
\item either $s\sqsubseteq t$, or 
\item $s\perp t$ and $s_{s\wedge t}\prec t_{s\wedge t}$.
\end{itemize}
Then $\langle T,\leq\rangle$ is a linearly ordered set and we may consider it as a LOTS. It is well-known that if $\langle T,\sqsubseteq\rangle$ is a Souslin tree, then $\langle T,\leq\rangle$ is a Souslin line.

\begin{thm}\label{thm:souslin_line_fc}
A Souslin line defined from a normal Souslin tree is functionally countable and thus, exponentially separable.
\end{thm}
\begin{proof}
Let $\langle T,\sqsubseteq\rangle$ be a normal Souslin tree and denote its tree topology by $\tau$. Let $\prec$ be any well-order of $T$ and let $\leq$ be the lexicographic order defined as described above. Let $\sigma$ be the LOTS topology of $\langle T,\leq\rangle$.

We will prove that the  identity $\mathsf{id}\colon\langle T,\tau\rangle\to\langle T,\sigma\rangle$ is continuous. In \cite{steprans-trees} it was proved that $\langle T,\tau\rangle$ is functionally countable. Since functional countability is clearly preserved by continuous images, it follows that $\langle T,\sigma\rangle$ is also functionally countable. So it is enough to prove that if $t\in T$, then
$$
\begin{array}{lcl}
U & = & \{s\in T\colon t<s\}\textrm{ and}\\
V & = & \{s\in T\colon s<t\}
\end{array}
$$
are open in the tree topology.

First we prove that $U\in \tau$; let $s\in U$. If $t\sqsubset s$, then $t{\uparrow}$ is a $\tau$-open set such that $s\in t{\uparrow}\subset U$. Otherwise, $s\perp t$ and $t_{s\wedge t}\prec s_{s\wedge t}$. In this case, since $T$ is normal there exists $u\in T$ such that $s\sqsubset u$. Let $W=\left[u{\downarrow}\right]\cap\left[(s\wedge t){\uparrow}\right]$; notice that $W$ is a $\tau$-open set. Since $s\wedge t\sqsubset s_{s\wedge t}\sqsubseteq s$, it follows that $s\in W$. Now, given $w\in W$, we will show that $w\in U$. Recall that by the definition of a tree, $u{\downarrow}$ is well-ordered. In $u{\downarrow}$, the immediate succesor of $s\wedge t$ is $s_{s\wedge t}$. Therefore, $s_{s\wedge t}\sqsubseteq w$, which implies that $w\perp t$, $w\wedge t=s\wedge t$, $w_{w\wedge t}=s_{s\wedge t}$ and $t_{w\wedge t}=t_{s\wedge t}$. Thus, $t_{w\wedge t}\prec w_{w\wedge t}$, which implies that $t<w$, that is, $w\in U$. We conclude that $s\in W\subset U$ for every $s\in U$. This implies that $U\in\tau$.

Now, let us argue why $V\in\tau$; let $s\in V$. If $s\sqsubset t$, then $t{\downarrow}$ is a $\tau$-open set such that $s\in t{\downarrow}\subset V$. Otherwise, $s\perp t$ and $s_{s\wedge t}\prec t_{s\wedge t}$. Since $T$ is normal, again there is $u\in T$ such that $s\sqsubset u$. Then, by an argument similar to the one in the previous paragraph, $W=\left[u{\downarrow}\right]\cap\left[(s\wedge t){\uparrow}\right]$ is a $\tau$-open set with $s\in W\subset V$.
\end{proof}

\section{Further examples of functionally countable and exponentially separable spaces}

In this section we explore various ways of constructing functionally countable and exponentially separable spaces, by using certain subspaces of products and the $G_\delta$ topology.

If $\prod_{\alpha\in A}{X_\alpha}$ is a Cartesian product and $B\subset A$, the projection $\pi_B\colon\prod_{\alpha\in A}{X_\alpha}\to \prod_{\alpha\in B}{X_\alpha}$ is defined by $\pi_B(f)=f\restriction B$ for all $f\in\prod_{\alpha\in A}{X_\alpha}$. We will need the following classical factorization result.

\begin{thm}[Factorization Lemma] \cite[0.2.3]{arkh-top_funct_spaces}
 Let $X=\prod_{\alpha\in A}{X_\alpha}$ be a product of spaces, each with a countable base, let $Y$ be dense in $X$ and $f\colon Y\to\R$ continuous. There is a countable set $B\subset A$ and a continuous function $g\colon \pi_B[Y]\to\R$ such that $f=g\circ \pi_B$.
\end{thm}

We start by showing that the members of a family of spaces with arbitrarily large extent constructed by Uspenski are functionally countable.

\begin{thm}[Uspenski's examples] \cite{uspenskij} \label{uspenskij-thm}
 Let $S$ be an infinite set. There exists a space $X_S$ with the following properties:
 \begin{enumerate}[label=(\roman*)]
  \item $X_S$ is a dense subset of $[0,1]^S$,
  \item $X_S$ is the union of a countable collection of closed discrete subspaces of size $\lvert S\rvert$,
  \item for all $f\in X_S$, $\{\alpha\in S\colon f(\alpha)\neq 0\}$ is finite, and
  \item there is a countable set $C\subset[0,1]$ such that for all $f\in X_S$, $f[S]\subset C$.
 \end{enumerate}
\end{thm}

\begin{thm}
For each uncountable cardinal $\kappa$, Uspenskii's example of cardinality $\kappa$ is a functionally countable space with a $G_\delta$ diagonal and extent $\kappa$, thus it is not exponentially separable.
\end{thm}
\begin{proof}
Let $X=X_{\kappa}$ be the space by Uspenskii from Theorem \ref{uspenskij-thm}. Clearly, points of $X$ are $G_\delta$ by property (ii). Also, $\sigma$-closed discrete spaces have a $G_\delta$ diagonal. Let $f\colon X\to\R$ be a continuous function. By the Factorization Lemma, there is a countable set $A\subset\kappa$ and a continuous function $g\colon \pi_A[X]\to\R$ such that $f=g\circ\pi_A$. Without loss of generality, assume that $A=\omega$. Let $Y=\pi_\omega[X]$; by properties (iii) and (iv) of $X$ we know that there exists a countable set $C\subset[0,1]$ and such that
$$
Y\subset\bigcup_{n\in\omega}{\left(\prod_{m<n}{C}\times\prod_{m\geq n}\{0\}\right)}
$$
Thus, $Y$ is countable. Then $g[Y]=f[X]$ is countable.
\end{proof}

The following theorem gives some insight on why certain classes of spaces are exponentially separable. Its proof is a variation of an argument due to Tkachuk.

Recall that given a space $X$ we denote by $X_\delta$ the $G_\delta$ topology on $X$, that is, the topology on $X$ generated by the $G_\delta$ subsets of $X$.

\begin{thm} \label{gdeltathm}
Let $X$ be a Lindel\"of space. If $X_\delta$ is also Lindel\"of then $X$ is exponentially separable.
\end{thm}

\begin{proof}
Let $\mathcal{F}$ be a countable family of closed subsets of $X$. For every $x \in X$ define $\mathcal{F}_x=\{F \in \mathcal{F}: x \notin F \}$. Since $\mathcal{F}_x$ is countable, we can find a closed $G_\delta$ subset $G_x$ containing $x$ such that $G_x \cap (\bigcup \mathcal{F}_x)=\emptyset$. Note that $\{G_x: x \in X \}$ is an open cover of $X_\delta$ and hence, since $X_\delta$ is Lindel\"of, we can find a countable subset $C$ of $X$ such that $X=\bigcup \{G_x: x \in C \}$. We claim that $C$ is strongly dense in $\mathcal{F}$. Indeed, let $\mathcal{H}$ be a subfamily of $\mathcal{F}$ such that $\bigcap \mathcal{H} \neq \emptyset$ and fix a point $y \in \bigcap \mathcal{H}$. Let $x \in C$ be such that $y \in G_x$ and observe that $F \cap G_x \neq \emptyset$, for every $F \in \mathcal{H}$, which yields that $x \in F$ for every $F \in \mathcal{H}$, and hence $(\bigcap \mathcal{H}) \cap C \neq \emptyset$, as needed.
\end{proof}

The following is now obvious:

\begin{coro}
(Tkachuk, \cite{tka-nice_subclass}) Let $X$ be a Lindel\"of $P$-space. Then $X$ is exponentially separable.
\end{coro}

\begin{coro}
(Tkachuk, \cite{tka-nice_subclass}) Let $X$ be a Lindel\"of scattered space. Then $X$ is exponentially separable.
\end{coro}

\begin{proof}
By a result of Uspenskii (see \cite{uspenskijGdelta}) if $X$ is a Lindel\"of scattered space then $X_\delta$ is Lindel\"of.
\end{proof}

\begin{coro}
(Tkachuk, \cite{tka-nice_subclass}) The space $X=\sigma(2^\kappa):=\{x \in 2^\kappa: |x^{-1}(1)| < \aleph_0 \}$ with the topology induced from $2^\kappa$ is exponentially separable.
\end{coro}

\begin{proof}
It is well-known that $X_\delta$ is Lindel\"of. That follows, from example, from the fact that $X_n=\{x \in 2^{\kappa}: |x^{-1}(1)| \leq n \}$ is a compact scattered space and hence $(X_n)_\delta$ is a Lindel\"of $P$-space (see \cite[Lemma II.7.14, p. 86]{arkh-top_funct_spaces}). It turns out that $X_\delta$ is the countable union of Lindel\"of $P$-spaces and hence Lindel\"of.
\end{proof}

Recall that the \emph{point-open game} on the space $X$ is the game between two players where at inning $n<\omega$ player I chooses a point $x_n \in X$ and player II chooses an open neighbourhood $U_n \subset X$ of $x_n$, and player I wins if $\{U_n: n < \omega \}$ is an open cover of $X$. A space $X$ is called \emph{finite-like} if player I has a winning strategy when the point-open game is played on $X$.

\begin{coro}
(Tkachuk, \cite{tka-applications_exp_sep}) Let $X$ be a \emph{finite-like space}. Then $X$ is exponentially separable.
\end{coro}

\begin{proof}
In Proposition 2.15 the authors of \cite{AR} proved that if player II has a winning strategy in the Rothberger game when played on $X$ then $X_\delta$ is Lindel\"of. The result now follows from Theorem $\ref{gdeltathm}$ because the Rothberger game and the point-open game are dual.
\end{proof}

Note that Justin Moore's example of an $L$-space (see \cite{moore}) is uncountable exponentially separable and has points $G_\delta$. Therefore, the converse of Theorem $\ref{gdeltathm}$ is not true.

We would like to finish this section with some remarks on spaces of continuous functions with the pointwise convergence topology. As noticed by Tkachuk in \cite{tka-Cp2023}, spaces of continuous real-valued functions $C_p(X)$ are never functionally countable: for any $x\in X$ the evaluation map $f\mapsto f(x)$ is continuous and its image is all of $\R$. However, if instead of considering $C_p(X)$ we take $C_p(X,2)$, that is, the space of continuous functions from a zero-dimensional space $X$ to the two point discrete space, the situation is more interesting. For example, in \cite{tka-corson_lc_scattered} Tkachuk proved the following theorem.

\begin{thm}\cite[Theorem 3.21]{tka-corson_lc_scattered}\label{thm:tka-CpX2}
 Let $X$ be a zero-dimensional countably compact Tychonoff space. Then the following are equivalent:
 \begin{enumerate}[label=(\alph*)]
 \item $C_p(X)$ has a dense functionally countable subspace,
  \item $C_p(X,2)$ is functionally countable, and
 \item $X$ is $\omega$-monolithic.
 \end{enumerate}
\end{thm}

A notable class of compact spaces that are $\omega$-monolithic is that of Corson compacta, that is compact subspaces of $\Sigma$-products of lines. In the same paper Tkachuk provided the following characterization.

\begin{thm}\cite[Theorem 3.18]{tka-corson_lc_scattered} \label{Tkachukthm}
Let $K$ be a zero-dimensional compactum. The following are equivalent

\begin{enumerate}
\item $K$ is a Corson compactum.
\item  $C_p(K,2)$ is finite-like. 
\item $C_p(K,2)$ is lc-scattered.
\end{enumerate}
\end{thm}

By an lc-scattered space we mean a space that is the continuous image of a Lindelöf scattered space (\cite[Definition 3.1]{tka-corson_lc_scattered}). Since every space which is either finite-like or lc-scattered is exponentially separable, it follows that, if $X$ is a zero-dimensional Corson compactum, then $C_p(X,2)$ is exponentially separable.



Eberlein compacta, that is, weakly compact subspaces of Banach spaces, are a notable subclass of the class of Corson compacta. If $X$ is a zero-dimensional Eberlein compactum, then $C_p(X,2)$ is exponentially separable in a very strong way. In \cite[Theorem IV.6.4]{arkh-top_funct_spaces} it was proved that for compact zero dimensional spaces $X$, $C_p(X,2)$ is $\sigma$-compact if and only if $X$ is an Eberlein compactum. Moreover, assuming that $X$ is an Eberlein compactum, the proof presented in that reference actually shows that $C_p(X,2)$ is a countable union of compact scattered spaces. We present an alternative short proof of this fact by making use of Theorem $\ref{Tkachukthm}$.

\begin{coro}
Let $X$ be a compact zero-dimensional space. Then the following are equivalent
\begin{enumerate}[label=(\alph*)]
\item $X$ is an Eberlein compactum,
\item $C_p(X,2)$ is a countable union of compact and scattered spaces, and
\item $C_p(X,2)$ is $\sigma$-compact.
\end{enumerate}  
\end{coro}
\begin{proof}
We only prove that (a) implies (b). That (b) implies (c) is straightforward and the proof that (a) is equivalent to (c) is contained in \cite[Theorem IV.6.4]{arkh-top_funct_spaces}.

Assume that $X$ is an Eberlein compactum. In \cite[Theorem IV.6.4]{arkh-top_funct_spaces} it is proved that  $C_p(X,2)=\bigcup_{n\in\omega}{F_n}$ where $F_n$ is compact for each $n\in\omega$. Notice that by Theorem $\ref{Tkachukthm}$ we know that $C_p(X,2)$ is exponentially separable so $F_n$ is compact and exponentially separable for all $n\in\omega$. Since compact and exponentially separable spaces are scattered (\cite[Corollary 3.14]{tka-nice_subclass}), we obtain (b). 
\end{proof}

The following diagram gives a summary of the implications found so far for compact zero-dimensional spaces.\vskip10pt
{\scriptsize
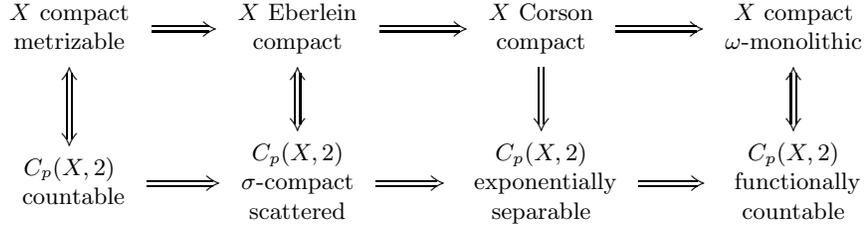
\begin{center}
\xymatrix{
{\begin{array}{c} X \textrm{ compact}\\ \textrm{metrizable} \end{array}} \ar@{=>}[r]\ar@{<=>}[d] & {\begin{array}{c} X \textrm{ Eberlein}\\  \textrm{compact} \end{array}} \ar@{=>}[r]\ar@{<=>}[d] & {\begin{array}{c} X \textrm{ Corson}\\ \textrm{compact} \end{array}}  \ar@{=>}[r] \ar@{=>}[d] & {\begin{array}{c} X \textrm{ compact}\\ \omega\textrm{-monolithic} \end{array}} \ar@{<=>}[d]\\
{\begin{array}{c} C_p(X,2)\\ \textrm{ countable}\end{array}} \ar@{=>}[r] & {\begin{array}{c} C_p(X,2)\\ \sigma\textrm{-compact}\\ \textrm{scattered} \end{array}}\ar@{=>}[r]  & {\begin{array}{c}C_p(X,2)\\ \textrm{ exponentially}\\ \textrm{separable} \end{array}} \ar@{=>}[r]& {\begin{array}{c}C_p(X,2)\\ \textrm{ functionally}\\ \textrm{countable} \end{array}}
}
\end{center}
}

Notice that all vertical arrows but one in the above diagram go both ways. That suggests the following natural question

\begin{ques}
Let $X$ be a zero-dimensional compactum such that $C_p(X,2)$ is exponentially separable. Is it then true that $X$ is a Corson compactum?
\end{ques}

\section{Separable plus exponentially separable}

The only known example of an uncountable and functionally countable space that is separable that has been discussed in the recent literature is an Isbell-Mrowka $\Psi$-space whose Čech-Stone compactification coincides with its one-point compactification (see \cite[Example 3.15]{tka-nice_subclass}). However, since this example contains an uncountable discrete subspace (the almost disjoint family used to construct it), it cannot be exponentially separable.
We present the following example

\begin{thm}\label{ex:unctble_sep_exp_sep}
There exists an uncountable space that is separable and exponentially separable.
\end{thm}
\begin{proof}
It is well-known that there exists a family $\mathcal{A}=\{a_\alpha\colon\alpha<\omega_1\}$ of strictly almost-decreasing infinite subsets of $\omega$; to be precise, if $\alpha<\beta<\omega_1$ we have that $a_{\beta}\setminus a_\alpha$ is finite but $a_\alpha\setminus a_\beta$ is infinite. 

Let $X=\omega\cup\mathcal{A}$ with the following topology: points of $\omega\cup\{a_0\}$ are open and for every $\alpha<\beta<\omega_1$ and $F\in[\omega]^{<\omega}$, $$\{a_\gamma\colon\alpha<\gamma\leq \beta\}\cup\left(a_\alpha\setminus (a_\beta\cup F)\right)$$ is a basic open neighborhood around $a_\beta$. It is not hard to prove that $X$ is a Tychonoff locally compact space, $\omega$ is dense in $X$ and $\mathcal{A}$ is canonically homeomorphic to $\omega_1$.

To prove that $X$ is exponentially separable, let $\{F_n\colon n<\omega\}$ be closed sets of $X$. Since $\mathcal{A}$ is homeomorphic to $\omega_1$, it is exponentially separable and we may find a countable set $\mathcal{C}\subset\mathcal{A}$ that is strongly dense in $\{F_n\cap\mathcal{A}\colon n<\omega\}$. Then it follows that $\mathcal{C}\cup\omega$ is a countable set that is strongly dense in $\{F_n\colon n<\omega\}$.
\end{proof}

The space $X$ we constructed in Theorem \ref{ex:unctble_sep_exp_sep} is the classical Franklin-Rajagopalan space (see, for example, \cite{levy}). As it is well-known (for example, see \cite{nyikos-vaughan} or \cite[Section 7]{vd-integers_topology}),  we can replace $\mathcal{A}$ with a tower of infinite subsets of $\omega$ and define the topology in an analogous way; in this case the resulting space $X$ is countably compact.

Under some additional set-theoretic hypotheses it is possible to construct an example with stronger properties, as shown by the next result.

\begin{thm}\label{thm:ostaszewski}
An Ostaszewski space is an example of a perfectly normal, hereditarily separable and exponentially separable uncountable space.
\end{thm}
\begin{proof}
Let $X$ be an Ostaszewski space. As it is well-known, $X=\omega_1$ as sets but its topology has the following properties: $X$ is hereditarily separable, locally countable, locally compact and each of its closed subsets is either countable or co-countable. We will prove that $X$ is exponentially separable.

Let $\{F_n\colon n<\omega\}$ be a collection of closed subsets of $X$. Define $A=\{n<\omega\colon F_n\textrm{ is countable}\}$ and $B=\omega\setminus A$. Then we may find $\eta<\omega_1$ such that $F_n\subset\eta$ for all $n\in A$ and $\omega_1\setminus\eta\subset F_n$ for all $n\in B$. Then it easily follows that $\eta+1$ is a countable subset that is strongly dense in $\{F_n\colon n<\omega\}$.

That $X$ is perfectly normal follows from the fact that every closed subset is countable or co-countable.
\end{proof}

Notice that the space $X$ from Theorem \ref{ex:unctble_sep_exp_sep} is not perfectly normal since the set $\{a_{\alpha}\colon\alpha<\omega_1,\alpha\textrm{ is a limit}\}$ is closed in $X$ but it is not a $G_\delta$; also, $X$ is not hereditarily separable. Thus, the following questions seem natural.

\begin{ques}
Is there in \ZFC{} an example of a perfectly normal, separable exponentially separable uncountable space?
\end{ques}

\begin{ques}
Is there in \ZFC{} an example of a hereditarily separable, exponentially separable uncountable space?
\end{ques}

%
%
%

To finish, we ask a question about Isbell-Mrowka $\Psi$-spaces. Arguing as in \cite[Example 3.15]{tka-nice_subclass} it can be seen that every Isbell-Mrowka $\Psi$-space whose Čech-Stone remainder is scattered will yield an example of a functionally countable space which is separable and uncountable. We can ask if this is the only way we can obtain such an example.

\begin{ques}
Is there a functionally countable Isbell-Mrowka $\Psi$-space whose Čech-Stone remainder is crowded, that is, has no isolated points?
\end{ques}

\section*{Acknowledgements}
Research of the first-named author was supported by a grant of the GNSAGA group of INdAM and the CONAHCyT grants FORDECYT-PRONACES/64356/2020 and CBF2023-2024-2028. The second-named author was supported by a grant from the GNSAGA group of INdAM and by a grant from the Fondo Finalizzato alla Ricerca di Ateneo (FFR 2024) of the University of Palermo.


\begin{thebibliography}{99}

 \bibitem{arkh-top_funct_spaces} Arkhangelskii, A.V. Topological Function Spaces. Mathematics and its Applications, Soviet Series. 78. Dordrecht etc.: Kluwer Academic Publishers. ix, 205 p. (1992).

\bibitem{AR} Aurichi L.F. and Dias R.R.; \emph{Topological Games and Alster Spaces}, Canad. Math. Bull. \textbf{57} (2014), 683--696.

 \bibitem{vd-integers_topology} van Douwen, E. K.; \emph{The integers and topology.} In: Handbook of set-theoretic topology, 111-167 (1984).
 
 \bibitem{gruenhage-handbook} Gruenhage, G.; \emph{Generalized metric spaces.} In: Handbook of set-theoretic topology, 423--501 (1984).
 
 \bibitem{gruenhage-gulko} Gruenhage, G.; \emph{A note on Gul’ko compact spaces.} Proc. Am. Math. Soc. 100, 371-376 (1987).

\bibitem{GD-HG} Gutiérrez-Domínguez, L. E.; Hernández-Gutiérrez, R.; \emph{When is the complement of the diagonal of a LOTS functionally countable?} Acta Math. Hung. 169, No. 2, 534--552 (2023).

\bibitem{ency} Hart, Klaas Pieter (ed.); Nagata, Jun-iti (ed.); Vaughan, Jerry E. (ed.); Encyclopedia of general topology. Amsterdam: Elsevier (ISBN 0-444-50355-2/hbk). x, 526 p. (2004).

\bibitem{kunen-set_theory_2011} Kunen, K.; Set theory. Studies in Logic (London) 34. London: College Publications. viii, 401 p., (2011). ISBN 978-1-84890-050-9.

\bibitem{levy} Levy, R.; \emph{Pseudocompactness and extension of functions in Franklin-Rajagopalan spaces.} Topology Appl. 11, 297--303 (1980).

\bibitem{lutzer} Lutzer,  D.J.; \emph{Ordered topological spaces.} In: Surveys in General Topology, Academic Press, 1980, pp. 247--295.

\bibitem{moore} Moore, J. T.; \emph{A solution to the L space problem.} J. Am. Math. Soc. 19, No. 3, 717--736 (2006).

\bibitem{nyikos-vaughan} Nyikos, P. J.; Vaughan, J. E.; \emph{Sequentially compact Franklin-Rajagopalan spaces.} Proc. Am. Math. Soc. 101, No. 1, 149--155 (1987).

\bibitem{purisch-compactifications} Purisch, S.; 
\emph{Scattered compactifications and the orderability of scattered spaces. II.} Proc. Am. Math. Soc. 95, 636-640 (1985).

\bibitem{purisch-orderability_closed_images} Purisch, S.;  \emph{The orderability and closed images of scattered spaces.} Trans. Am. Math. Soc. 320, No. 2, 713--725 (1990). 

\bibitem{steprans-trees} Steprāns, Juris; \emph{Trees and continuous mappings into the real line.} Topology Appl. 12, 181--185 (1981).

\bibitem{tka-LSigma} Tkachuk, V. V.; \emph{Lindelöf $\Sigma$-spaces: an omnipresent class.} Rev. R. Acad. Cienc. Exactas Fís. Nat., Ser. A Mat., RACSAM 104, No. 2, 221-244 (2010). 

\bibitem{tka-Sokolov} Tkachuk, V.V.; \emph{Lindel\"of $P$-spaces need not be Sokolov}, Math. Slovaca \textbf{67} (2017), 227--234.


\bibitem{tka-corson_funct_ctable} Tkachuk, V. V.; \emph{A Corson compact space is countable if the complement of its diagonal is functionally countable.}  Stud. Sci. Math. Hung. 58, No. 3, 398--407 (2021). 
 
 \bibitem{tka-nice_subclass} Tkachuk, V.V.; \emph{A nice subclass of functionally countable spaces.} Comment. Math. Univ. Carolin. 59, No. 3, 399--409 (2018).
 
 \bibitem{tka-applications_exp_sep} Tkachuk, V. V.; \emph{Some applications of exponentially separable spaces.} Quaest. Math. 43, No. 10, 1391--1403 (2020). 

 \bibitem{tka-func_ctble_some_products} Tkachuk, V.V.; \emph{Functional countability is preserved by some products.} Acta Math. Hungar. 167, No. 2, 612--622 (2022).

 \bibitem{tka_wil-fc_GO} Tkachuk, V.V.; Wilson, R.G.; \emph{Functional countability in GO spaces.} Topology Appl. 320, 108233 (2022).
 
\bibitem{tka-Cp2023} Tkachuk, V.V.; \emph{Cp-theory in 2022.} Quest. Answers Gen. Topology 41, No. 2, 75--133 (2023).

\bibitem{tka-corson_lc_scattered} Tkachuk, V. V.; \emph{A compact space $K$ is Corson compact if and only if $C_p(K)$ has a dense lc-scattered subspace.} J. Math. Anal. Appl. 533, No. 1, Article ID 127992, 11 p. (2024).
 
\bibitem{todorcevic-handbook} Todorčević, S.; \emph{Trees and linearly ordered sets.} Handbook of set-theoretic topology, 235--293 (1984). 

\bibitem{uspenskijGdelta} Uspenskij V.V.; \emph{On the spectrum of frequencies of function spaces}  Vestnik Moskov Univ. Ser. I Mat. Mekh \textbf{37} (1982), no. 1, 31--35.
 
 \bibitem{uspenskij} Uspenskij, V.V.; \emph{A large $F\sigma$-discrete Fréchet space having the Souslin property.}  Commentat. Math. Univ. Carol. 25, 257--260 (1984).
\end{thebibliography}
\end{document}